\documentclass[a4paper,reqno]{amsart}
\usepackage{amscd, amssymb, pb-diagram, amsmath, amscd}
\usepackage{amsthm}
\usepackage{slashed}

\usepackage{tikz}

\usepackage{tikz-cd}
\usetikzlibrary{calc}
\usepackage[margin=3.7cm]{geometry}
\usepackage{color}
\usepackage{tensor}
\usetikzlibrary{arrows,matrix,backgrounds,decorations.markings}
\numberwithin{equation}{section}

\usepackage[all]{xy}

\DeclareFontFamily{OT1}{pzc}{}
\DeclareFontShape{OT1}{pzc}{m}{it}{<-> s * [1.2] pzcmi7t}{}
\DeclareMathAlphabet{\mathpzc}{OT1}{pzc}{m}{it}

\def\End{\operatorname{End}}

\def\ker{\operatorname{Ker}}

\def\id{\operatorname{id}}

\def\id{\operatorname{id}}

\def\C{\mathbb{C}}

\def\R{\mathbb{R}}

\def\Z{\mathbb{Z}}

\def\Q{\mathbb{Q}}

\newtheorem{thm*}{Theorem}
\newtheorem{thm}{Theorem}[section]

\newtheorem{prop}[thm]{Proposition}

\theoremstyle{definition}
\newtheorem{definition}[thm]{Definition}

\theoremstyle{remark}
\newtheorem{remark}[thm]{Remark}

\newtheorem{example}[thm]{Example}

\begin{document}

\date{\today}
\title{Fell algebras, groupoids, and projections}

\author{Robin J. Deeley}
\address{Robin J. Deeley,   Department of Mathematics,
University of Colorado Boulder
Campus Box 395,
Boulder, CO 80309-0395, USA }
\email{robin.deeley@colorado.edu}
\author{Magnus Goffeng}
\address{Magnus Goffeng,
Centre for Mathematical Sciences,
University of Lund,
Box 118, 221 00 LUND, Sweden
}
\email{magnus.goffeng@math.lth.se}
\author{Allan Yashinski}
\address{Allan Yashinski,  Department of Mathematics, University of Maryland, College Park, MD 20742-4015, USA }
\email{ayashins@umd.edu}

\thanks{RJD is currently funded by NSF Grant DMS 2000057 and was previously funded by Simons Foundation Collaboration Grant for Mathematicians number 638449. MG was supported by the Swedish Research Council Grant VR 2018-0350.}

\begin{abstract}
Examples of Fell algebras with compact spectrum and trivial Dixmier-Douady invariant are constructed to illustrate differences with the case of continuous-trace $C^*$-algebras. At the level of the spectrum, this translates to only assuming the spectrum is locally Hausdorff (rather than Hausdorff). The existence of (full) projections is the fundamental question considered. The class of Fell algebras studied here arise naturally in the study of Wieler solenoids and applications to dynamical systems will be discussed in a separate paper.
\end{abstract}

\maketitle

\section{Introduction}
Being type I $C^*$-algebras, Fell algebras form a quite well understood class of $C^*$-algebras. Their basic properties are classical and are discussed in detail in \cite{Dix:C*}. The reader familiar with continuous-trace algebras can note that a Fell algebra with Hausdorff spectrum is a continuous-trace algebra. Nevertheless there has been renewed interest in Fell algebras from the perspective of groupoid $C^*$-algebras recently in \cite{CHR, HKS}. This note fits within this line of inquiry. However, our original motivation came from dynamical systems where replacing a compact Hausdorff space by a compact locally Hausdorff space ``resolves'' complicated dynamical behaviour. 

For context we briefly discuss the connection between Fell algebras and Smale space $C^*$-algebras, even if Smale spaces will not play any role in the mathematical content of this note. A Smale space is a certain dynamical system \cite{Rue} to which there are a number of $C^*$-algebras associated, e.g., the stable algebra, the Ruelle algebras, etc (see \cite{Put} for details). In the special case of a Wieler solenoid, the first and third listed authors proved that there is a Fell algebra that plays a fundamental role in understanding the stable algebra and the Ruelle algebras. Exactly what these algebras are is not overly relevant; the point is that results in \cite{DeeYas} reduce the problem of understanding them (and by extension the dynamics of the Wieler solenoid) to understanding the dynamics of a particularly nice Fell algebra, one with compact spectrum and trivial Dixmier-Douady invariant. 

Noticing this situation, the authors of the present paper were hopeful that the class of Fell algebras with compact spectrum and trivial Dixmier-Douady invariant would be sufficiently well-behaved to completely understand the stable algebra and Ruelle algebras in the case of Wieler solenoids, see for example the introduction of \cite{DGY}. By sufficiently well-behaved, we mean that results from the continuous-trace case would generalize without much issue or change to our setting. For instance, one could think that a Fell algebra with compact spectrum and trivial Dixmier-Douady invariant is stably unital (i.e., admits a full projection), and at least when it is stably unital then it is Morita equivalent to a unital algebra associated to a nice groupoid, and finally that if additionally the spectrum is a non-Hausdorff manifold it would share $K$-theoretical features satisfied by a classical manifold. The main goal of this note is to present examples showing that our optimism was misplaced and none of the statements in the previous sentence are true. 

Let us consider the case of continuous-trace $C^*$-algebras with compact spectrum and trivial Dixmier-Douady invariant. The spectrum of such an algebra will be denoted by $X$. In this case the spectrum is Hausdorff and the Morita equivalence class of such an algebra only depends on the spectrum. Furthermore, $C(X)$ is a unital representative in the Morita equivalence class and $C(X) \otimes \mathcal{K}$ is the stable representative in the Morita equivalence class. There is nothing particularly deep about the results in the previous sentence but we encourage the reader to take a moment to appreciate how deeply satisfying the situation is in this case.

Moving to the Fell algebra case (still with compact spectrum and trivial Dixmier-Douady invariant) one has that the spectrum is locally Hausdorff. In \cite{HKS}, it is shown that if $A$ is such an algebra then $A\otimes \mathcal{K}$ admits an explicit groupoid model (see Example \ref{fellex} for the precise construction). They also show that the Morita equivalence class of such an algebra again only depends on the spectrum. As such, the stable algebra in the Morita equivalence class has an explicit model as a groupoid $C^*$-algebra; this is analogous to taking $C(X)\otimes \mathcal{K}$ as the representative in the continuous-trace case.

Our difficulties occur when we ask for a nice unital groupoid model in the Morita equivalence class of such a Fell algebra; this would be analogous to taking $C(X)$ in the continuous-trace case. The term ``nice groupoid model'' is formalized below in Definition \ref{nicegroupoidmodel}. Furthermore, one could ask for additional $K$-theoretical properties of ``manifold-like" Fell algebras pertaining to the existence of noncommutative geometries realizing the Fell algebra geometrically. In this note we provide examples of the following types:
\begin{enumerate}
\item a Fell algebra with compact spectrum, trivial Dixmier-Douady invariant but no non-zero projection (see Section \ref{BrokenHeartEx});
\item a Fell algebra with compact spectrum, trivial Dixmier-Douady invariant, containing non-zero projections but no full projection (see Section \ref{brokenheartofsoleniods});
\item a Fell algebra with compact spectrum, trivial Dixmier-Douady invariant, a full projection but the Morita equivalence class of this algebra does not contain a unital algebra with a nice groupoid model in the sense of Definition \ref{nicegroupoidmodel} (see Section \ref{sec:twistedsphere});
\item a Fell algebra whose spectrum is a compact one-dimensional smooth non-Hausdorff manifold, admitting a surjective local homeomorphism from the circle, has trivial Dixmier-Douady invariant and is Poincar\'e self-dual of even dimension but not of odd dimension (see Section \ref{aababSolenoidFellAlgebra});
\item a Fell algebra whose spectrum is a compact smooth non-Hausdorff manifold, admitting a surjective local homeomorphism from a smooth compact manifold, has trivial Dixmier-Douady invariant, but has infinitely generated $K$-theory (see Section \ref{infgenktheoru})
\end{enumerate}

In summary, if $A$ is a Fell algebra with compact spectrum and trivial Dixmier-Douday invariant, then these examples imply that there might not be any unital representative in the Morita equivalence class of $A$ and even if there is a unital representative it might not have a nice groupoid model. And even when there is a nice groupoid model, the Fell algebra could behave in a $K$-theoretically unexpected way.

The paper is structured as follows. Section \ref{generalstuff} discusses the general setup and how the fundamental results from \cite{CHR, HKS} fit within it. Theorems related to the existence of projections are presented in Section \ref{sometheoremsmsms} and the examples of the various types above are presented in the remaining sections. Originally we were planning to include these examples in a paper about the Fell algebras associated with Wieler solenoids. However, we believe these examples are of independent interest and reemphasize that no knowledge of dynamical systems is required to read the present paper.

\subsection*{Acknowledgements} The authors wishes to thank the anonymous referee for their careful reading and helpful comments.

\section{Fell algebras and their spectrum}
\label{generalstuff}
For the benefit of the reader, we will recall the basic facts about $C^*$-algebras we make use of. None of the results are new and most can be found in \cite{CHR,Dix:C*,HKS}. We have given explicit references whenever possible. For a separable $C^*$-algebra $A$, the spectrum $\hat{A}$ is defined as the set of equivalence classes of irreducible representations (see \cite[Chapter 3]{Dix:C*}). The spectrum $\hat{A}$ can be topologized in several equivalent ways, for instance in the Fell topology or the Jacobson topology. By \cite[Corollary 3.3.8]{Dix:C*}, the spectrum $\hat{A}$ is locally quasi-compact.

We are mainly concerned with Fell algebras. A $C^*$-algebra $A$ is a Fell algebra if every $[\pi_0]\in \hat{A}$ admits a neighborhood $U$ and an element $b\in A$ such that $\pi(b)$ is a rank one projection for all $[\pi]\in U$. Equivalently, $A$ is generated by its abelian elements. See more in \cite[Chapter 3]{HKS}. By \cite[Corollary 3.4]{archsom}, the spectrum of a Fell algebra is locally Hausdorff (i.e., any $[\pi]\in \hat{A}$ has a Hausdorff neighborhood). Since the spectrum of a $C^*$-algebra is locally quasi-compact, we summarize the properties of the spectrum of a Fell algebra as being locally Hausdorff and locally locally compact (see \cite[Chapter 3]{CHR}). 

\begin{example}
\label{fellex}
The Fell algebras that we will concern ourselves with arise in a rather explicit way. The construction can be found in \cite[Corollary 5.4]{CHR}. Suppose $Y$ is a second countable, locally compact Hausdorff space and $\psi:Y\to X$ a surjective local homeomorphism onto a topological space $X$. It follows that $X$ is second countable, locally Hausdorff and locally locally compact. We define the equivalence groupoid:
$$R(\psi):=Y\times_\psi Y:=\{(y_1,y_2)\in Y\times Y: \psi(y_1)=\psi(y_2)\}.$$
We topologize $R(\psi)$ as a groupoid over $Y$ by taking the subspace topology (relative to $R(\psi) \subseteq Y\times Y$). Since $\psi$ is a surjective local homeomorphism, the domain and range maps $d(y_1,y_2):=y_2$ and $r(y_1,y_2):=r_1$ (respectively) are local homeomorphisms. Hence the groupoid $R(\psi)$ is \'etale. Following \cite{CHR}, we form the $C^*$-algebra $C^*(R(\psi))$; we note that in general there are a number of different $C^*$-algebras associated with an \'etale groupoid but  \cite[Theorem A.4]{CHR} implies that they are naturally isomorphic in our rather special situation. Furthermore, by \cite[Theorem 6.1]{CHR}, $C^*(R(\psi))$ is a Fell algebra with vanishing Dixmier-Douady invariant and spectrum $X$. 
\end{example}

The theory of Dixmier-Douady invariants of Fell algebras was introduced and developed in \cite{HKS}, also see \cite{CHR}. We only need to consider Fell algebras with vanishing Dixmier-Douady class in which case the following theorem (see \cite{CHR, HKS}) reduces the problem to a more manageable situation.

\begin{thm}
\label{fellchara}
Let $A$ be a separable Fell algebra with vanishing Dixmier-Douady invariant. Then the locally Hausdorff and locally locally compact space $\hat{A}$ determines $A$ up to stable isomorphism in the sense that whenever $A'$ is a separable Fell algebra with vanishing Dixmier-Douady invariant then a homeomorphism $h:\hat{A}\to \widehat{A'}$ can be lifted to a stable isomorphism $A\otimes \mathbb{K}\cong A'\otimes \mathbb{K}$.
\end{thm}

For a proof of Theorem \ref{fellchara}, see \cite[Theorem 7.13]{HKS}. While Theorem \ref{fellchara} seems to indicate a functorial relationship, we remark that spectra does not define a functor on all $*$-homomorphisms of Fell algebras. The failure can be seen already on the subcategory of finite-dimensional $C^*$-algebras. Our questions concerning projections can be reformulated as follows: does there exist a $*$-homomorphism $\mathbb{K}\to A\otimes \mathbb{K}$? Notice that for continuous-trace algebras with compact spectrum and trivial Dixmier-Douady invariant, one obtains this $*$-homomorphisms from the map at the level of spectrum given by crushing $X$ to a point. Resolving functoriality issues for Fell algebras poses an interesting problem for future work.

\begin{thm}
\label{fellcharb}
Let $X$ be a second countable, locally Hausdorff and locally locally compact space. Then there is a second countable, locally compact Hausdorff space and a surjective local homeomorphism $\psi:Y\to X$. In particular, $X$ is the spectrum of the separable Fell algebra $C^*(R(\psi))$ with vanishing Dixmier-Douady invariant.
\end{thm}

Theorem \ref{fellcharb} can be found in \cite[Corollary 5.5]{CHR}. We shall recall the details of the latter in the special case of compact $X$ because it is relevant for our constructions.

\begin{prop}
\label{constructiveprop}
Let $X$ be a second countable, compact, locally Hausdorff space and $\mathfrak{U}=(U_j)_{j=1}^N$ a finite cover of $X$ consisting of open Hausdorff subsets. Define the space $Z_\mathfrak{U}:=\coprod_{j=1}^N U_j$ and the mapping
$$\psi_\mathcal{U}:Z_\mathfrak{U}\to X, \quad\mbox{by declaring}\quad \psi_\mathfrak{U}|_{U_j}=\mathrm{id}_{U_j}.$$
Then $Z_\mathfrak{U}$ is a locally compact Hausdorff space and $\psi_\mathfrak{U}$ is a local homeomorphism. In particular, for any Fell algebra $A$ with vanishing Dixmier-Douady invariant and $\hat{A}\cong X$, there is a stable isomorphism $C^*(R(\psi_{\mathfrak{U}}))\otimes \mathbb{K}\cong A\otimes \mathbb{K}$.
\end{prop}

\begin{proof}
Since each set $U_j$ is a locally compact Hausdorff space, it follows that $Z_\mathfrak{U}$ also is. Any point in $Z_\mathfrak{U}$ belongs to some $U_j$ and since $\psi_\mathfrak{U}|_{U_j}$ is a homeomorphism onto its open range, $\psi_\mathfrak{U}$ is a local homeomorphism. The conclusion of the proposition follows from Example \ref{fellex} (see also \cite[Corollary 5.4]{CHR}) and Theorem \ref{fellchara} (see also \cite[Theorem 7.13]{HKS}).
\end{proof}

It is of interest to note that using Theorem \ref{fellchara} and \ref{fellcharb}, and the process in the proof of Proposition \ref{constructiveprop}, one has that if $A$ is a stable Fell algebra with trivial Dixmier-Douady invariant, then it is isomorphic to the groupoid $C^*$-algebra associated to a local homeomorphism as in Example \ref{fellex}. In other words, the Morita equivalence class of a Fell algebra with trivial Dixmier-Douady invariant and spectrum $X$ always contains a representative that is a groupoid $C^*$-algebra obtained from a local homeomorphism $\psi:Y\to X$. The representative defined from $\psi:Y\to X$ is unital if and only if $Y$ is compact. In fact the results in \cite{HKS} are more general; they apply to non-trivial Dixmier-Douady invariant if one allows twisted groupoid $C^*$-algebras obtained from a local homeomorphism. Based on this observation, we make the following definition.

\begin{definition}
\label{nicegroupoidmodel}
Suppose that $X$ is a second countable, locally Hausdorff and compact space. A \emph{nice groupoid model} for $X$ is a groupoid of the form $R(\psi)$ where $\psi:Y\to X$ a surjective local homeomorphism onto $X$ and $Y$ is a compact Hausdorff space.
\end{definition}

The following theorem is a special case of item (3) on page 342 of \cite{Bro}, which gives a complete description of when a continuous-trace $C^*$-algebra is stably isomorphic to a unital one. We state this special case so as to contrast the situation between continuous-trace $C^*$-algebras and Fell algebras obtained from local homeomorphisms.

\begin{thm} \label{conTraceCase}
Suppose that $C^*(R(q))$ is the $C^*$-algebra obtained from a local homeomorphism $q: W \rightarrow Z$ where $W$ is locally compact and Hausdorff and $Z$ is compact and connected. If $Z$ is Hausdorff, then $C^*(R(q))$ is stably isomorphic to $C(Z)$.
\end{thm}

The theorems in the next section and the examples discussed in the succeeding sections illustrate that the situation for Fell algebras (i.e., when $Z$ is locally Hausdorff rather than Hausdorff) is more involved. 

\section{Theorems on the existence of projections in Fell algebras}
\label{sometheoremsmsms}

\begin{thm} 
\label{ExistenceProjectionOpenCompactHaus}
Suppose that $A$ is a Fell algebra with trivial Dixmier-Douady invariant and spectrum $X$. If there exists $K$ a nonempty, compact, open, Hausdorff subset of $X$, then $A \otimes \mathcal{K}$ contains a non-zero projection.
\end{thm}

\begin{proof}
The existence of non-zero projections in $A \otimes \mathcal{K}$ is a Morita invariant statement. Up to Morita equivalence, $A=C^*(R(\psi))$ for some local homeomorphism $\psi: Y \to X$. Upon replacing $Y$ with $Y \sqcup K$ we can assume that $K\subseteq Y$ and that $\psi|_K$ is the identity. Consider $K$ as the trivial groupoid over itself, with domain and range maps being the identity map, so $C^*(K)=C(K)$. We have a clopen groupoid inclusion $K\subseteq R(\psi)$ that induces a faithful $*$-homomorphism $C(K)\to C^*(R(\psi))$. Since $K$ is compact $C(K)$ contains a projection, namely the unit, so its image in $C^*(R(\psi))$ is a non-trivial projection.
\end{proof}

\begin{thm}
Suppose that $A$ is a Fell algebra with trivial Dixmier-Douady invariant and its spectrum $X:=\hat{A}$ admits a surjective local homemorphism $\psi:Y\to X$ where $Y$ is compact and Hausdorff. Then $A$ is stably isomorphic to a unital Fell algebra. 
\end{thm}

\begin{proof}
We can apply the construction in Example \ref{fellex} to $\psi: Y \rightarrow X$. The resulting groupoid $C^*$-algebra is unital because its unit space $Y$ is compact. Furthermore, Theorem \ref{fellchara} implies that $A$ is stably isomorphic to this groupoid $C^*$-algebra.
\end{proof}

\section{The $aab/ab$-solenoid}
\label{aababSolenoidFellAlgebra}
An example for which both the previous theorems are relevant was considered in \cite{DeeYas}. Although this example is self-contained, the Fell algebra is related to aab/ab-solenoid, which is a Smale space and the reader can see \cite{DeeYas} for further context for this example. 

The compact, locally Hausdorff space, $X$, is pictured in Figure \ref{Figure-aab/ab-QuotientSpace}. Informally, it is formed by taking the wedge of two circles and then spliting the wedge point into three non-Hausdorff points, denoted $ab, ba, aa$. Open Hausdorff neighborhoods for these three points are pictured in Figure \ref{Figure-aab/ab-OpenNeighborhoods}. 
	
	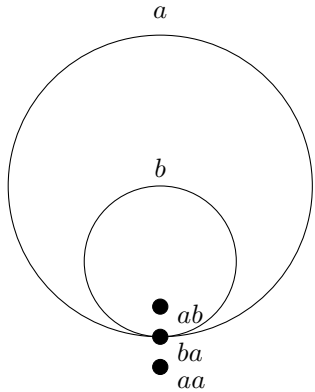
\begin{figure}[h]
		\begin{tikzpicture}
		\draw (0,0) circle [radius=2];
		\draw (0,-1) circle [radius=1];
		\draw[fill] (0,-1.6) circle [radius=0.1];
		\draw[fill] (0,-2) circle [radius=0.1];
		\draw[fill] (0,-2.4) circle [radius=0.1];
		\node [right] at (0.1,-1.7) {$ab$};
		\node [right] at (0.1,-2.2) {$ba$};
		\node [right] at (0.1,-2.6) {$aa$};
		\node [above] at (0,2.1) {$a$};
		\node [above] at (0,0) {$b$};
		\end{tikzpicture}
		\caption{The compact, locally Hausdorff space, $X$.}
		\label{Figure-aab/ab-QuotientSpace}
	\end{figure}
	
	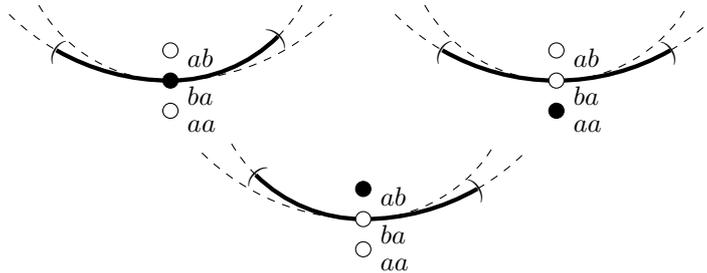
\begin{figure}[h]
		\begin{tikzpicture}
		\draw [dashed] ($(0,0) + (225:3)$) arc (225:240:3);
		\draw [ultra thick] ($(0,0) + (240:3)$) arc (240:270:3);
		\draw [dashed] ($(0,0) + (270:3)$) arc (270:315:3);
		\draw [dashed] ($(0,-1) + (210:2)$) arc (210:270:2);
		\draw [ultra thick] ($(0,-1) + (270:2)$) arc (270:315:2);
		\draw [dashed] ($(0,-1) + (315:2)$) arc (315:330:2);
		\node [rotate=-30] at (-1.5,-2.61) {$($};
		\node [rotate=45] at (1.414,-2.414) {$)$};
		\draw (0,-2.6) circle [radius=0.1];
		\draw[fill] (0,-3) circle [radius=0.1];
		\draw (0,-3.4) circle [radius=0.1];
		\node [right] at (0.1,-2.7) {$ab$};
		\node [right] at (0.1,-3.2) {$ba$};
		\node [right] at (0.1,-3.6) {$aa$};
		\end{tikzpicture} \qquad
		\begin{tikzpicture}
		\draw [dashed] ($(0,0) + (225:3)$) arc (225:240:3);
		\draw [dashed] ($(0,0) + (300:3)$) arc (300:315:3);
		\draw [ultra thick] ($(0,0) + (240:3)$) arc (240:270:3);
		\draw [ultra thick] ($(0,0) + (270:3)$) arc (270:300:3);
		\draw [dashed] ($(0,-1) + (210:2)$) arc (210:270:2);
		\draw [dashed] ($(0,-1) + (270:2)$) arc (270:330:2);
		\node [rotate=-30] at (-1.5,-2.61) {$($};
		\node [rotate=30] at (1.5,-2.61) {$)$};
		\draw (0,-2.6) circle [radius=0.1];
		\draw[fill=white] (0,-3) circle [radius=0.1];
		\draw[fill] (0,-3.4) circle [radius=0.1];
		\node [right] at (0.1,-2.7) {$ab$};
		\node [right] at (0.1,-3.2) {$ba$};
		\node [right] at (0.1,-3.6) {$aa$};
		\end{tikzpicture} \qquad
		\begin{tikzpicture}
		\draw [dashed] ($(0,-1) + (210:2)$) arc (210:225:2);
		\draw [dashed] ($(0,0) + (300:3)$) arc (300:315:3);
		\draw [dashed] ($(0,0) + (225:3)$) arc (225:270:3);
		\draw [ultra thick] ($(0,0) + (270:3)$) arc (270:300:3);
		\draw [ultra thick] ($(0,-1) + (225:2)$) arc (225:270:2);
		\draw [dashed] ($(0,-1) + (270:2)$) arc (270:330:2);
		\node [rotate=-45] at (-1.414,-2.414) {$($};
		\node [rotate=30] at (1.5,-2.61) {$)$};
		\draw[fill] (0,-2.6) circle [radius=0.1];
		\draw[fill=white] (0,-3) circle [radius=0.1];
		\draw (0,-3.4) circle [radius=0.1];
		\node [right] at (0.1,-2.7) {$ab$};
		\node [right] at (0.1,-3.2) {$ba$};
		\node [right] at (0.1,-3.6) {$aa$};
		\end{tikzpicture}
		\caption{Open neighborhoods of the three non-Hausdorff points in $X$.}
		\label{Figure-aab/ab-OpenNeighborhoods}
	\end{figure}

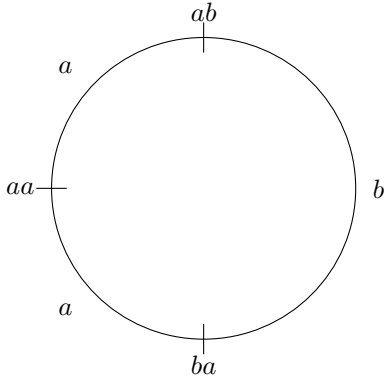
\begin{figure}[h]
\begin{tikzpicture}
\draw (0,0) circle [radius=2];
\draw (0, -2.2) -- (0, -1.8);
\draw (0,1.8) -- (0,2.2);
\draw (-2.2,0) -- (-1.8,0);
\node [below] at (0,-2.1) {$ba$};
\node [above] at (0,2.1) {$ab$};
\node [left] at (-2.1,0) {$aa$};
\node [right] at (2.1,0) {$b$};
\node [left] at (-1.6,-1.6) {$a$};
\node [left] at (-1.6,1.6) {$a$};
\end{tikzpicture}
\caption{The circle with labels determining the map to $X$}
\label{Figure-aab/ab-Circle cover}
\end{figure}

There is a surjective local homeomorphism from the circle to the space $X$. The map is defined using Figure \ref{Figure-aab/ab-Circle cover}. Each open interval in the circle labeled with a maps homeomorphically to the outer circle with the wedge points removed. While the open interval in the circle labeled b maps homeomorphically to the inner circle with the wedge points removed. The image of the other three points is given by labels (which are ab, ba, aa). One can check that the quotient topology obtained from this map gives the topology that was informally defined using Figure \ref{Figure-aab/ab-OpenNeighborhoods}. Since the circle is compact, we see that the previous theorem can be applied to this example. 

In addition, the subset of $X$ given by the outer circle union the point labeled aa is a nonempty, compact, open, Hausdorff subset. Using Theorem \ref{ExistenceProjectionOpenCompactHaus} there is a projection associated to this subset. Summarizing the Fell algebra associated to the local homeomorphism from the circle to $X$ is unital and also contains a non-full, non-zero projection.

\begin{remark}
As seen above, many compact locally Hausdorff spaces admits a surjective local homeomorphism from a compact Hausdorff space. There are however compact, locally Hausdorff spaces which do not. An explicit example is the twisted sphere constructed in \cite[Section 3]{hommelbergbthe}, we discuss it below in Section \ref{sec:twistedsphere}.
\end{remark}

We proceed by extending a $K$-theory computation for the compact, locally Hausdorff space $X$ pictured in Figure \ref{Figure-aab/ab-QuotientSpace} from \cite{DeeYas}. This computation shows that general non-Hausdorff manifolds, even with nice groupoid models, need not share basic $K$-theoretical features satisfied by classical manifolds. The reader can compare this to the Poincar\'e duality results for non-Hausdorff manifolds constructed from simplicial complexes in \cite{KasSkand}. 

\begin{prop}
Let $X$ be the compact, locally Hausdorff space pictured in Figure \ref{Figure-aab/ab-QuotientSpace}. We set $K_*(X):=K^*(A)$ and $K^*(X):=K_*(A)$ for some Fell algebra $A$ with spectrum $X$. Then there are isomorphisms
$$K_0(X)\cong K^0(X)\cong \Z^2 \quad\mbox{and}\quad  K_1(X)\cong K^1(X)\cong \Z.$$
Moreover, $A$ is Poincar\'e self-dual of even dimension and even though $X$ is a smooth compact odd-dimensional non-Hausdorff manifold it is not $KK$-isomorphic to any odd-dimensional smooth compact manifold.
\end{prop}

That $A$ is Poincar\'e self-dual of even dimension means that for any $C^*$-algebras $B$ and $C$ there are natural isomorphisms 
$$KK_*(A\otimes B,C)\cong KK_*(B,A\otimes C).$$

\begin{proof}
By the results of \cite{DeeYas}, the stabilization of $A$ fits into a short exact sequence 
\begin{equation}
\label{sesforaaabab}
0\to C_0(\R)\otimes \C^2\otimes \mathbb{K}\to A\otimes \mathbb{K}\to \C^3\otimes \mathbb{K}\to 0.
\end{equation}
The map $A\otimes \mathbb{K}\to \C^3\otimes \mathbb{K}$ is given by evaluation in the three points $ab$, $ba$ and $aa$, cf. Figure \ref{Figure-aab/ab-QuotientSpace}. It was computed in \cite{DeeYas} that the associated boundary mapping on $K$-theory $K_0(\C^3\otimes \mathbb{K})\to K_1(C_0(\R)\otimes \C^2\otimes \mathbb{K})$ is represented by the matrix 
$$\delta_0=
\begin{pmatrix}
-1&1&0\\
1&-1&0
\end{pmatrix},
$$
under the isomorphisms $K_0(\mathbb{K})\cong \Z$ and $K_1(C_0(\mathbb{R}))\cong \Z$. From this a computation with six term exact sequences showed $K_0(A)\cong \Z^2$ and $K_1(A)\cong \Z$ in \cite{DeeYas}. 

The short exact sequence \eqref{sesforaaabab} implies that $A$ is nuclear and in the bootstrap class. In particular, the universal coefficient theorem allows us to deduce that the associated boundary mapping on $K$-homology $K^1(C_0(\R)\otimes \C^2\otimes \mathbb{K})\to K^0(\C^3\otimes \mathbb{K})$ is represented by the matrix 
$$\delta_0^T=
\begin{pmatrix}
-1&1\\
1&-1\\
0&0
\end{pmatrix},
$$
under the isomorphisms $K^0(\mathbb{K})\cong \Z$ and $K^1(C_0(\mathbb{R}))\cong \Z$. A similar computation as in \cite{DeeYas} proves that $K^0(A)\cong\mathrm{coker}(\delta_0^T:\Z^2\to \Z^3)\cong \Z^2$ and $K^1(A)\cong \ker(\delta_0^T:\Z^2\to \Z^3)\cong \Z$. 

Since we have isomorphisms $K_0(A)\cong K^0(A)$ and $K_1(A)\cong K^1(A)$, and $A$ is in the bootstrap category it follows that $A$ is Poincar\'e self-dual of even dimension. The final statement of the proposition follows from that $K^1(A)\otimes \Q\not\cong K_0(A)\otimes \Q$, so $A$ can not be rationally Poincar\'e self-dual of odd dimension but for any smooth compact manifold $M$, $C(M)$ is rationally Poincar\'e self-dual of odd dimension.
\end{proof}

\section{The broken heart} 
\label{BrokenHeartEx} 

We construct a Fell algebra with compact spectrum and trivial Dixmier-Douady invariant but no non-zero projection. The Morita equivalence classes of this algebra and all its ideals do not contain a unital representative.

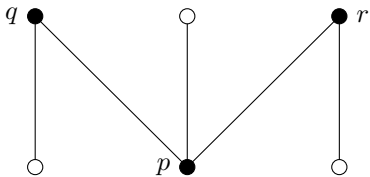
\begin{figure}[h]
	\begin{tikzpicture}
	\draw (0,0) -- (0,2);
	\draw (0,0) -- (-2, 2) -- (-2, 0);
	\draw (0,0) -- (2, 2) -- (2, 0);
	\draw[fill] (0,0) circle [radius=0.1];
	\draw[fill] (-2,2) circle [radius=0.1];
	\draw[fill] (2,2) circle [radius=0.1];
	\draw[fill=white] (0,2) circle [radius=0.1];
	\draw[fill=white] (-2,0) circle [radius=0.1];
	\draw[fill=white] (2,0) circle [radius=0.1];
	\node [left] at (-0.1,0) {$p$};
	\node [left] at (-2.1,2) {$q$};
	\node [right] at (2.1,2) {$r$};
	\end{tikzpicture}
	\caption{$Y$ for broken heart}
	\label{Figure-Y-BrokenHeartAlt}
\end{figure}

Let $Y$ be the space in Figure \ref{Figure-Y-BrokenHeartAlt}. Define an equivalence relation on $Y$ by identifying the three open vertical segments in Figure \ref{Figure-Y-BrokenHeartAlt}.

\begin{figure}[h]
	\begin{tikzpicture}
	\draw (0,0) -- (0,2);
	\draw (0,0) -- (-2, 2) -- (-2, 0);
	\draw (0,0) -- (2, 2) -- (2, 0);
	\draw[fill] (0,0) circle [radius=0.1];
	\draw[fill] (-2,2) circle [radius=0.1];
	\draw[fill] (2,2) circle [radius=0.1];
	\draw[fill=white] (0,2) circle [radius=0.1];
	
	\draw[fill] (0,1) circle [radius=0.1];
	\draw[fill] (-2,1) circle [radius=0.1];
	\draw[fill] (2,1) circle [radius=0.1];
	
	\draw[fill] (1,1) circle [radius=0.1];
	\draw[fill] (-1,1) circle [radius=0.1];

	\draw[fill=white] (-2,0) circle [radius=0.1];
	\draw[fill=white] (2,0) circle [radius=0.1];
	\node [left] at (-0.1,0) {$p$};
	\node [left] at (0,1) {$a_2$};
	\node [left] at (-2,1) {$a_1$};
	\node [left] at (2,1) {$a_3$};
	\node [left] at (1,1) {$y$};
	\node [left] at (-1,1) {$x$};
	\node [left] at (-2.1,2) {$q$};
	\node [right] at (2.1,2) {$r$};
	\end{tikzpicture}
	\caption{$Y$ for broken heart with representative equivalent classes}
	\label{Figure-Y-BrokenHeartWithEquiv}
\end{figure}
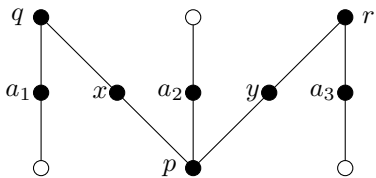

In Figure \ref{Figure-Y-BrokenHeartWithEquiv}, a number of equivalent classes are denoted; the points labelled with $a_1$, $a_2$ and $a_3$ form one equivalence class, the point labelled with an $x$ forms another, the point labelled with a $y$ forms yet another, etc.

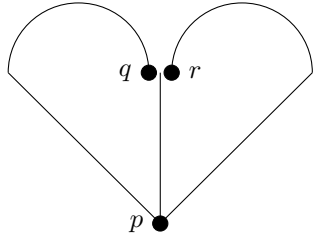
\begin{figure}[h]
	\begin{tikzpicture}
	\draw (0,0) -- (0,2);
	\draw (0,0) -- (-2, 2);
	\draw (0,0) -- (2, 2);
	\draw (-2,2) arc (180:0:0.925);
	\draw (2,2) arc (0:180:0.925);
	\draw[fill] (0,0) circle [radius=0.1];
	\draw[fill] (-0.15,2) circle [radius=0.1];
	\draw[fill] (0.15,2) circle [radius=0.1];
	\node [left] at (-0.1,0) {$p$};
	\node [left] at (-0.25,2) {$q$};
	\node [right] at (0.25,2) {$r$};
	\end{tikzpicture}
	\caption{$X$ for broken heart}
	\label{Figure-X-BrokenHeart}
\end{figure}

The quotient space associated to this equivalence relation will be denoted by $X$, see Figure \ref{Figure-X-BrokenHeart}. We let $A$ denote the Fell algebra with trivial Dixmier-Douady invariant associated with the surjective local homeomorphism $Y\to X$ defined from the quotient.

One way to show that this example admits no projections is to use a six-term exact sequence to prove that $K_0(A)\cong 0$. The detail of this $K$-theory computations are similar to ones in \cite{DeeYas} so we only provide an outline. There is an ideal of $A$ obtained from the union of the three open vertical line segments, see Figure \ref{Figure-Y-BrokenHeartAlt}. This ideal is isomorphic to $C_0((0,1))\otimes M_3(\C)$. The quotient of $A$ by this ideal is isomorphic to $C([0,1])$. Leading to the short exact sequence of $C^*$-algebras
\[
0 \rightarrow C_0((0,1))\otimes M_3(\C) \rightarrow A \rightarrow C([0,1]) \rightarrow 0.
\]  
The six-term exact sequence of $K$-theory groups reduces to
\[ 0 \rightarrow K_0(A) \rightarrow \Z \rightarrow \Z \rightarrow K_1(A) \rightarrow 0.
\]
Then, one checks that the boundary map from $\Z$ to $\Z$ is an isomorphism. In particular, we have that $K_0(A) \cong 0$.

Finally, since $K_0(A)\cong 0$, it follows that $A \otimes \mathbb{K}$ has no nonzero projections, because such a projection would have to be nonzero in some irreducible representation. It is worth noting that this example implies that the result \cite[bottom of page 342, item (3)]{Bro} cannot be generalized to Fell algebras (even with trivial Dixmier-Douady invariant).

\section{The broken heart of the $aab/ab$-solenoid}
\label{brokenheartofsoleniods}

We combine the examples of Sections \ref{aababSolenoidFellAlgebra} and \ref{BrokenHeartEx} to obtain a Fell algebra with compact spectrum, trivial Dixmier-Douady invariant, no full projection but with a nonzero projection. The Morita equivalence class of this Fell algebra does not contain a unital representative. 

Let $\pi_1: Y_1 \rightarrow X_1$ and $\pi_2: Y_2 \rightarrow X_2$ be as in Section \ref{aababSolenoidFellAlgebra} and \ref{BrokenHeartEx}, respectively. Taking the wedge product of $Y_1$ and $Y_2$ as in Figure \ref{Figure-wedgeOfTwoSpaces}, one obtains a local homeomorphism $\pi : Y_1 \vee Y_2 \rightarrow  X_1 \vee X_2$. It is important to note that the choice of wedge point is not arbitrary; we have taken it to be on the arc labelled with a $b$ in $Y_1$ and the point labelled by $p$ in $Y_2$. In particular, using this choice of wedge point ensures that the map $\pi$ is a local homeomorphism. 

\begin{figure}[h]
	\begin{tikzpicture}
	\draw (0,0) -- (0,2);
	\draw (0,0) -- (-2, 2) -- (-2, 0);
	\draw (0,0) -- (2, 2) -- (2, 0);
	\draw[fill] (0,0) circle [radius=0.1];
	\draw[fill] (-2,2) circle [radius=0.1];
	\draw[fill] (2,2) circle [radius=0.1];
	\draw[fill=white] (0,2) circle [radius=0.1];
	\draw[fill=white] (-2,0) circle [radius=0.1];
	\draw[fill=white] (2,0) circle [radius=0.1];
	\node [below] at (0,-0.1) {$p$};
	\node [left] at (-2.1,2) {$q$};
	\node [right] at (2.1,2) {$r$};
	
\draw (0,-2) circle [radius=2];
\draw (0, -4.2) -- (0, -3.8);
\draw (1.8, -2) -- (2.2, -2);
\draw (-2.2,-2) -- (-1.8,-2);

\node [below] at (0, -4.1) {$aa$};
\node [left] at (-2.1,-2) {$ab$};
\node [right] at (2.1,-2) {$ba$};
\node [left] at (-1.6,-3.6) {$a$};
\node [right] at (1.6,-3.6) {$a$};
	\end{tikzpicture}
	\caption{The space $Y_1 \vee Y_2$}
	\label{Figure-wedgeOfTwoSpaces}
\end{figure}
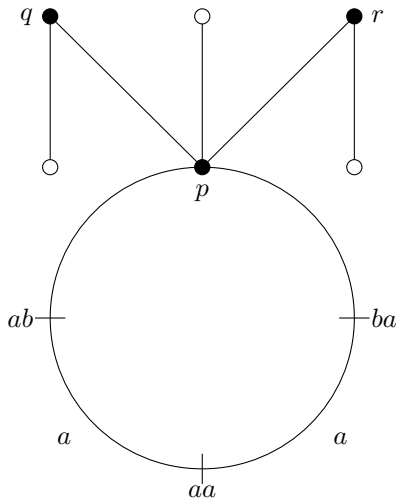

Following the notation in Example \ref{fellex}, the associated Fell algebra is denoted by $C^*(R(\pi))$. Its unit space is $Y_1\vee Y_2$ and its spectrum is $X_1 \vee X_2$ where the wedge point is on the $b$-circle of $X_1$ (see Figure \ref{Figure-aab/ab-QuotientSpace}) and the point labelled $p$ in $X_2$ (see Figure \ref{Figure-X-BrokenHeart}). It follows from Theorem \ref{ExistenceProjectionOpenCompactHaus} with $K\subseteq X_1$ being the $a$-circle crossing through $aa$ (in the same way as in Section \ref{aababSolenoidFellAlgebra}) that $C^*(R(\pi))$ contains a non-zero projection.

We now show that $C^*(R(\pi))$ does not contain a full projection. As a subset of $Y_1 \vee Y_2$, $Y_2$ is closed and is invariant with respect to the groupoid $R(\pi)$ (that is, if $(w,z)\in R(\pi)$ then $w \in Y_2$ if and only if $z\in Y_2$). These two conditions imply that the restriction map defined on compactly supported functions extends to a $*$-homomorphism $\rho : C^*(R(\pi)) \rightarrow C^*(R(\pi_2))$. Since $C^*(R(\pi_2))$ contains only the zero projection, it follows that if $p$ is a projection in $C^*(R(\pi))$ then $\rho(p)=0$. It follows from this observation and \cite[Lemma 1.1]{Bro} that $C^*(R(\pi))$ does not contain a full projection.

\section{The twisted sphere}
\label{sec:twistedsphere}

Next we construct a Fell algebra with compact spectrum, trivial Dixmier-Douady invariant, a full projection but its spectrum does not admit a surjective local homeomorphism from a compact Hausdorff space. As a result, there is no nice groupoid model constructed from a local homeomorphism in the Morita equivalence class of this Fell algebra. 

This example is based on a topological space constructed by Hommelberg in \cite[Section 3]{hommelbergbthe}. It might be useful for the reader to have a copy of \cite{hommelbergbthe} on hand when considering the construction below, see in particular Figures 3.1 and 3.2 (c) on pages 10 and 11 of \cite{hommelbergbthe}.

Consider the equivalence relation $\sim$ on $S^2 = \{ (x,y,z) \in \R^3 ~:~ x^2 + y^2 + z^2 = 1\}$ generated by
\[ (x,y,z) \sim (-x,-y,z) \qquad \text{if } z \neq 0. \]  Let $X$ be the quotient space $S^2/{\sim}$, which we will call the ``twisted sphere".  $X$ is like $S^2$, except it has a ``long equator" that wraps around the sphere twice.  The space $X$ is a compact, non-Hausdorff manifold.  In particular, it is a locally Hausdorff space.  The quotient map $q: S^2 \to X$ is not nice for our purposes, because it is not a local homeomorphism at the poles.  Hence we cannot use it to construct a Fell algebra whose spectrum is $X$. In fact, one of the main results in \cite{hommelbergbthe} is the following:
\begin{thm}(Proposition 3.3.4 in \cite{hommelbergbthe}) \label{noLocalHomeoTwisted} There is no surjective local homeomorphism from a compact Hausdorff space to the space $X$ defined in the previous paragraph.
\end{thm}

We now return to the construction of the Fell algebra(s) with spectrum $X$. Let $U = S^2\setminus \{ (0,0,1), (0,0,-1)\}$. Then $q: U \to X$ is a local homeomorphism.  We need to also cover the poles of $X$ using a local homeomorphism.  Let
\[ E = \{(x,y,z) \in S^2 ~:~ z=0\} \]
denote the equator of $S^2$, so that $q(E)$ is the equator of $X$.  Let
\[ V = S^2 \setminus E.\]
We will frequently use cylindrical coordinates to refer to points on $S^2$.  Each point on $S^2$ other than the poles has unique cylindrical coordinates $(\theta, z)$ such that $0 \leq \theta < 2\pi$.  Using these unique coordinates, consider the (discontinuous) map
\[ \alpha: V \to S^2,\qquad \alpha(\theta, z) = (\theta/2, z), \] which also maps the north pole to the north pole and the south pole to the south pole.  Note that the composition
\[ q \circ \alpha: V \to X \]
is a continuous and one-to-one map of $V$ onto $X\setminus E$. One can show that $q \circ \alpha$ is a local homeomorphism; we omit the details.

We can use the two local homeomorphisms $q: U \to X$ and $q \circ \alpha: V \to X$ to construct a surjective local homeomorphism from $U \sqcup V$ to $X$.  From this, we can construct a groupoid whose $C^*$-algebra is a Fell algebra whose spectrum is $X$.  It can be shown that this Fell algebra does not have a nonzero projection.  

However, using a different space $Y$ and local homeomorphism, one can construct a different groupoid $C^*$-algebra for this space that does have a nonzero projection. The construction proceeds as follows.

Let $U$ and $V$ be as above.  Let $V_1 = V$ and $V_2 = V$ be two copies of $V$.  Let
\[ Y = U \sqcup V_1 \sqcup V_2, \]
which is a locally compact Hausdorff space. In fact, it is a (non-connected) manifold. We define a local homeomorphism
\[ \psi: Y \to X \] by
\[ \psi\rvert_U = q: U \to X,\qquad \psi\rvert_{V_i} = q\circ \alpha: V_i \to X. \]

Let us describe the equivalence relation $\sim$ induced by $\psi$.  Since all elements of $U, V_1, V_2$ are elements of $S^2$, we shall try to avoid confusion by denoting elements of $U$ by $u$ and elements of $V_1$ and $V_2$ by $v_1$ and $v_2$ respectively.  Given an element $u = (x,y,z)$ of $U$, let $\overline{u} = (-x,-y,z) \in U$.  Equivalently, if $u$ is given by $(\theta, z)$ in cylindrical coordinates, then $\overline{u}$ is given by $(\theta+\pi, z)$.  Assuming $u \notin E$, then $q(u) = q(\overline{u})$ where recall that $E = \{(x,y,z) \in S^2 ~:~ z=0\}$.  
\begin{itemize}
	\item A point $u \in E \subseteq U$ is only equivalent to itself.  We have $[u] = \{u\}$.
	\item A point $u \in U\setminus E$ has four members in its equivalence class.  We have $u \sim u$ and $u \sim \overline{u}$.  If $(\theta, z)$ are the cylindrical coordinates of $u \in U$, then let $v_i \in V_i = V$ be the point with cylindrical coordinates $(2\theta, z)$ (for $i=1,2$).  Note that $\alpha(v_i)$ equals either $u$ or $\overline{u}$.  Either case leads to the conclusion that $q(\alpha(v_i)) = q(u)$, so that $u \sim v_i$.  So $[u] = \{u, \overline{u}, v_1, v_2\}$.
	\item Suppose $v_1 \in V_i$ is one of the poles $(0,0,1)$ or $(0,0,-1)$, and $v_2$ is the same pole, but in the set $V_2$.  Then $[v_1] = \{v_1, v_2\}$.
\end{itemize}
This is a complete description of each equivalence class of the equivalence relation $\sim$ induced by $\psi$.  Note that there is one equivalence class for each point of $X$.

Let $A = C^*(R(\psi))$ be the groupoid $C^*$-algebra associated to this equivalence relation via the process in Example \ref{fellex}. As such, $A$ is a Fell algebra with trivial Dixmier-Douady invariant and spectrum $X$. 

We shall explicitly construct a full projection $p \in A$.  We define
\begin{itemize}
	\item $p(u,u) = 1-|z|$, where $u \in U$ has cylindrical coordinates $(\theta, z)$.
	\item $p(v_1, v_1) = |z|$, where $v_1 \in V_1 = V$ has cylindrical coordinates $(\theta, z)$.
	\item $p(v_2, v_2) = |z|$, where $v_2 \in V_2 = V$ has cylindrical coordinates $(\theta, z)$.
	\item $p(u,\overline{u}) = 0$ for any $u \in U\setminus E$.
	\item $p(v_1, v_2) = p(v_2, v_1) = 0$ for $v_1 \in V_1$ and $v_2 \in V_2$.
	\item $p(u, v_1) = \sqrt{\frac{1}{2}|z|(1-|z|)}$, where $u \in U$ has cylindrical coordinates $(\theta, z)$.  Note that if $u \sim v_1$, as we are assuming, then $z$ is also the cylindrical coordinate of $v_1 \in V_1$.
	\item $p(v_1, u) = \overline{p(u,v_1)} = \sqrt{\frac{1}{2}|z|(1-|z|)}$.
	\item $p(u, v_2) = e^{-i\theta}\sqrt{\frac{1}{2}|z|(1-|z|)}$, where $u \in U$ has cylindrical coordinates $(\theta, z)$.  Note that if $u \sim v_2$, as we are assuming, then $z$ is also the cylindrical coordinate of $v_2 \in V_2$.
	\item $p(v_2, u) = \overline{p(u,v_2)} = e^{i\theta}\sqrt{\frac{1}{2}|z|(1-|z|)}$, with $\theta, z$ as above.
\end{itemize}
By construction, $p(y,x) = \overline{p(x,y)}$ for all $x \sim y$, implying that $p$ is self-adjoint.  Next, we will show that $p * p = p$ where $*$ denotes the convolution product.
\begin{itemize}
	\item If $u \in E \subseteq U$, then $[u] = \{u\}$ and
	\[ (p * p)(u,u) = p(u,u)p(u,u) = (1-0)^2 = 1 = p(u,u). \]
	\item If $u \in U \setminus E$, then $[u] = \{u, \overline{u}, v_1, v_2\}$ and
	\begin{align*}
	(p * p)(u,u) &= p(u,u)p(u,u) + p(u,\overline{u})p(\overline{u},u) \\
	&\qquad\quad  + p(u,v_1)p(v_1,u) + p(u,v_2)p(v_2,u)\\
	&= |p(u,u)|^2 + 0 + |p(u,v_1)|^2 + |p(u,v_2)|^2\\
	&= (1-|z|)^2 + \frac{1}{2}|z|(1-|z|) + \frac{1}{2}|z|(1-|z|)\\
	&= (1 - |z|)(1 - |z|) + |z|(1 - |z|)\\
	&= 1 - |z|\\
	&= p(u,u)
	\end{align*}
	\item If $v_1 \in V_1$ is either the north or south pole, so that $[v_1] = \{v_1, v_2\}$ where $v_2$ is the same pole in $V_2$, then
	\begin{align*}
	(p * p)(v_1,v_1) &= p(v_1,v_1)p(v_1,v_1) + p(v_1,v_2)p(v_2,v_1)\\
	&= |p(v_1,v_1)|^2 = 1^2 = p(v_1,v_1).
	\end{align*}
	Similarly, $(p*p)(v_2,v_2) = p(v_2,v_2)$.
	\item Let $v_1 \in V_1$ be a point that is not the north or south pole.  Let $v_2$ be the matching point in $V_2$ and let $u \in U$ be some point equivalent to $v_1$, so that $[v_1] = \{u, \overline{u}, v_1, v_2\}$.  Note that the cylindrical $z$-coordinate is the same for each point in $[v_1]$.  So
	\begin{align*}
	(p * p)(v_1, v_1) &= p(v_1,u)p(u,v_1) + p(v_1, \overline{u})p(\overline{u}, v_1) \\
	&\qquad\quad + p(v_1, v_1)p(v_1, v_1) + p(v_1, v_2)p(v_2, v_1)\\
	&= |p(v_1,u)|^2 + |p(v_1, \overline{u})|^2 + |p(v_1, v_1)|^2 + 0\\
	&= \frac{1}{2}|z|(1-|z|) + \frac{1}{2}|z|(1-|z|) + |z|^2\\
	&= |z|(1-|z|) + |z|^2\\
	&= |z| = p(v_1, v_1).
	\end{align*}
	Similarly,
	\begin{align*}
	(p * p)(v_2, v_2) &= |p(v_2,u)|^2 + |p(v_2, \overline{u})|^2 + 0 + |p(v_2, v_2)|^2\\
	&= \frac{1}{2}|z|(1-|z|) + \frac{1}{2}|z|(1-|z|) + |z|^2\\
	&= |z| = p(v_2, v_2).
	\end{align*}
	This shows that $p*p = p$ on the unit space.
	\item Let $v_1 \in V_1$ be one of the poles, and let $v_2 \in V_2$ be the same pole.  Then
	\[(p * p)(v_1, v_2) = p(v_1, v_1)p(v_1,v_2) + p(v_1, v_2)p(v_2, v_2) = 0 + 0 = p(v_1, v_2).\]
\end{itemize}
For the remainder, let $u \in U \setminus E$, so that $[u] = \{u, \overline{u}, v_1, v_2\}$.  As before, all points have the same $z$-coordinate.  Note that if $\theta$ denotes the angle for $u$, then the angle for $\overline{u}$ is $\theta + \pi$.
\begin{itemize}
	\item 
	\begin{align*}
	(p * p)(u, \overline{u}) &= p(u,u)p(u,\overline{u}) + p(u,\overline{u})p(\overline{u},\overline{u}) \\
	&\qquad\quad + p(u,v_1)p(v_1,\overline{u}) + p(u,v_2)p(v_2,\overline{u})\\
	&= 0 + 0 + \sqrt{\frac{1}{2}|z|(1-|z|)}\sqrt{\frac{1}{2}|z|(1-|z|)}  \\
	&\qquad\quad + e^{-i\theta}\sqrt{\frac{1}{2}|z|(1-|z|)}e^{i(\theta + \pi)}\sqrt{\frac{1}{2}|z|(1-|z|)}\\
	&= \frac{1}{2}|z|(1-|z|) - \frac{1}{2}|z|(1-|z|)\\
	&= 0 = p(u, \overline{u}).
	\end{align*}
	\item
	\begin{align*}
	(p * p)(u, v_1) &= p(u,u)p(u,v_1) + p(u, \overline{u})p(\overline{u},v_1) \\
	&\qquad\quad + p(u,v_1)p(v_1,v_1) + p(u,v_2)p(v_2,v_1)\\
	&= (1-|z|)\sqrt{\frac{1}{2}|z|(1-|z|)} + 0 + \left(\sqrt{\frac{1}{2}|z|(1-|z|)}\right)|z| + 0\\
	&= \sqrt{\frac{1}{2}|z|(1-|z|)} = p(u, v_1).
	\end{align*}
	\item
	\begin{align*}
	(p * p)(u, v_2) &= p(u,u)p(u,v_2) + p(u, \overline{u})p(\overline{u},v_2) \\
	&\qquad\quad + p(u,v_1)p(v_1,v_2) + p(u,v_2)p(v_2,v_2)\\
	&= (1-|z|)e^{-i\theta}\sqrt{\frac{1}{2}|z|(1-|z|)} + 0 + 0 + \left(e^{-i\theta}\sqrt{\frac{1}{2}|z|(1-|z|)}\right)|z|\\
	&= e^{-i\theta}\sqrt{\frac{1}{2}|z|(1-|z|)} = p(u, v_2).
	\end{align*}
	\item
	\begin{align*}
	(p * p)(v_1, v_2) &= p(v_1,u)p(u,v_2) + p(v_1,\overline{u})p(\overline{u},v_2)  \\
	&\qquad\quad +p(v_1,v_1)p(v_1,v_2) + p(v_1,v_2)p(v_2,v_2)\\
	&= \sqrt{\frac{1}{2}|z|(1-|z|)}e^{-i\theta}\sqrt{\frac{1}{2}|z|(1-|z|)} \\
	&\qquad\quad + \sqrt{\frac{1}{2}|z|(1-|z|)}e^{-i(\theta+\pi)}\sqrt{\frac{1}{2}|z|(1-|z|)}+0+0\\
	&= e^{-i\theta}\frac{1}{2}|z|(1-|z|) - e^{-i\theta}\frac{1}{2}|z|(1-|z|)\\
	&= 0 = p(v_1, v_2).
	\end{align*}
\end{itemize}
This completes the proof that $p*p = p$ upon noticing that $p*p$ and $p$ are self-adjoint, so equalities of the form $(p*p)(v_1,u) = p(v_1,u)$ follow from the fact that $(p*p)(u,v_1) = p(u,v_1)$.

It is clear that $p$ is a nonzero projection. To see that it is full, notice for each $x\in X$ there exists $(y_1, y_2) \in R(\psi)$ such that $\psi(y_1)=\psi(y_2)=x$ and $p(y_1,y_2)\neq 0$. It then follows from the fact that $X$ is the spectrum of $C^*(R(\psi))$ and \cite[Lemma 1.1]{Bro} that $p$ is full. 

The $C^*$-algebra $p C^*(R(\psi))p$ is unital and since $p$ is full, it is Morita equivalent to $C^*(R(\psi))$. However, using the discussions just before Definition \ref{nicegroupoidmodel} and Theorem \ref{noLocalHomeoTwisted}, any unital Fell algebra Morita equivalent to $C^*(R(\psi))$ (in particular, $p C^*(R(\psi))p$) does not have a nice groupoid model in the sense of Definition \ref{nicegroupoidmodel}. 

\section{Compact non-Hausdoff manifolds with infinitely generated $K$-theory}
\label{infgenktheoru}

Finally we construct a Fell algebra with infinitely generated $K$-theory but with spectrum being a smooth compact non-Hausdorff manifold, more precisely it will be the image of a smooth compact connected manifold under a local homeomorphism. 

Our construction involves a more general setup (where the domain might be disconnected). Let $k$ be an integer strictly greater than one, $M$ be a compact, connected manifold, and $p: W \rightarrow M$ be a $k$-fold covering map. It follows that $W$ is a manifold but might not be connected (e.g., when $p$ is the trivial $k$-fold covering map). 

Let $U$ be an open set in $M$ that is evenly covered by $p$ so that $p^{-1}(U)$ is the disjoint union of open sets $(U_i)_{i=1}^k$ in $W$ where $p|_{U_i}$ is a homeomorphism onto $U$. Fix $A \subseteq U$ a closed subset of $M$ and notice that $p^{-1}(A)$ is the disjoint union of $(A_i)_{i=1}^k$ where $A_i=U_i\cap p^{-1}(A)$.  We define an equivalence relation on $W$ as follows:
\begin{enumerate}
\item $w \sim w$ for each $w\in W$;
\item $w \sim z$ when $p(w)=p(z) \not\in A$.
\end{enumerate}
Informally the quotient space associated to this equivalence relation (denoted by $W/\sim$) is obtained by taking $M$ and breaking each point in $A$ into $k$-points while leaving points in the complement alone. 

We define $X:=W/\sim.$ The construction satisfies the following:
\begin{enumerate}
\item The map $q: W \rightarrow X$ is a local homeomorphism;
\item $X$ is a compact non-Hausdorff manifold;
\item $X$ is independent of the choice of $k$-fold covering map of $M$ that evenly covers a neighborhood of $A$.
\end{enumerate} 
If $E\to M$ denotes the vector bundle defined from $C(M;E):=C(W)$ (with $C(M)$-action defined from $q$), the Fell algebra $C^*(R(q))$ can be identified with the algebra 
\[
\{ f \in C(M,\End(E)) \: | \: f(a) \text{ is diagonal for all }a \in A\}.
\]
The condition that $f(a)$ is diagonal for all $a \in A$ follows since $U$ is evenly covered by $p$, and the choice of open sets $(U_i)_{i=1}^k$ defines a trivialization $E|_U\cong U\times \C^k$ in which we make sense of the condition that $f(a)$ is diagonal. We set
$$K^*(X):=K_*(C^*(R(q))).$$

\begin{prop}
Let $X$ be the compact non-Hausdorff manifold constructed above from a $k$-fold covering of a compact manifold $M$. Then there is an isomorphism 
$$K^*(X)\cong K^*(M)\oplus K^*(A)^{k-1}.$$
In particular, $K^*(X)$ is infinitely generated if and only if $K^*(A)$ is.
\end{prop}

\begin{proof}
Since $X:=W/\sim$ is independent of the choice of $k$-fold covering map, the Fell algebras associated to two different choices of $k$-fold covers of $M$ are Morita equivalent. We can therefore assume that $q$ is the trivial covering and that the Fell algebra $C^*(R(q))$ is isomorphic to 
\[
\{f \in C(M,M_k(\C)) \: | \: f(a) \text{ is diagonal for all }a \in A\}
\]
We will use the notation $f_{ij}$ to refer to the entries of the matrix-valued function $f \in C^*(R(q))$.  Define a $*$-homomorphism
\[ \pi: C^*(R(q)) \to C(A)^{k-1},\qquad \pi(f)(a) = (f_{22}(a), f_{33}(a), \ldots , f_{kk}(a)). \]
Then $\pi$ is surjective and its kernel $I = \ker \pi$ consists of all $f \in C^*(R(q))$ with the property that $f(a)$ has the form 
$$f(a) = \begin{pmatrix}
f_{11}(a) & 0 & \ldots & 0\\
0 & 0 & \ldots & 0\\
\vdots & \vdots & \ddots & \vdots\\
0 & 0 & \ldots & 0
\end{pmatrix} \quad\mbox{whenever $a \in A$.}$$

We claim that the upper left corner inclusion $\iota: C(M) \to I$ induces an isomorphism on $K$-theory.  We have a commutative diagram of short exact sequences
\[ \xymatrix{
	0 \ar[r] &C_0(M\setminus A) \ar[r] \ar[d]_-{\iota} &C(M) \ar[r] \ar[d]_-{\iota} &C(A) \ar[r] \ar[d]_-{\id} & 0\\
	0 \ar[r] &C_0(M\setminus A, ~M_k(\C)) \ar[r] & I \ar[r] &C(A) \ar[r] & 0\\
} \]
where the maps onto $C(A)$ are restrictions and the map 
$$\iota: C_0(M\setminus A) \to C_0(M\setminus A, ~M_k(\C)),$$ 
is also the corner inclusion. The commutative diagram of $K$-theory groups is
\[ \xymatrix{
	\cdots \ar[r] &K_i(C_0(M\setminus A)) \ar[r] \ar[d]_-{\cong} &K_i(C(M)) \ar[r] \ar[d]_-{\iota_*} &K_i(C(A)) \ar[r] \ar[d]_-{\cong} & \cdots\\
\cdots \ar[r] &K_i(C_0(M\setminus A, M_k(\C))) \ar[r] & K_i(I) \ar[r] &K_i(C(A)) \ar[r] & \cdots\\
} \]
So $\iota_*: K_*(C(M))\to K_*(I)$ is an isomorphism by the Five Lemma.  Combining this with the six-term exact sequence induced by $\pi: C^*(R(q))\to C(A)^{k-1}$ gives a six-term exact sequence
\[ \xymatrix{
	K_0(C(M)) \ar[r] & K_0(C^*(R(q))) \ar[r] & K_0(C(A))^{k-1} \ar[d]\\
	K_1(C(A))^{k-1} \ar[u] & K_1(C^*(R(q))) \ar[l]& K_1(C(M)) \ar[l]
} \]
in which the map $\iota_*: K_*(C(M)) \to K_*(C^*(R(q)))$ is induced by inclusion into the upper left corner.  However, $\mu \circ \iota_* = \id$ where $\mu$ is the composition
\[ \mu: K_*(C^*(R(q))) \to K_*(C(M,M_k(\C))) \to K_*(C(M)) \]
of the map induced by inclusion and the inverse of the $K$-theory isomorphism induced by corner inclusion.  Thus the sequence is split-exact, and we obtain isomorphisms
\[ K^*(X) \cong K^*(M) \oplus K^*(A)^{k-1}. \]
Finally, since the manifold $M$ is compact, $K^*(M)$ is finitely generated and therefore  $K^*(X)$ is infinitely generated if and only if the space $A$ has infinitely generated $K$-theory.
\end{proof}

For an explicit example, one can take $W=M = \R/\Z$, $p$ the two-fold covering map of the circle by the circle and $A = \{1/n \: | \: n \geq 10\}\cup\{0\}$. The Fell algebra associated to this input will have non-finitely generated $K$-theory but its spectrum is a compact non-Hausdorff manifold. The unit space of the groupoid used to construct this Fell algebra is the circle. Many other examples can be constructed using this setup.

%\bibliography{Biblio-Database}
%\bibliographystyle{abbrv}

\end{document}